\documentclass[leqno,12pt]{amsart} 
\usepackage{amssymb}
\theoremstyle{plain} 
\newtheorem{theorem}{\indent\sc Theorem}[section]
\newtheorem{lemma}[theorem]{\indent\sc Lemma}
\newtheorem{corollary}[theorem]{\indent\sc Corollary}
\newtheorem{proposition}[theorem]{\indent\sc Proposition}

\theoremstyle{definition} 
\newtheorem{definition}[theorem]{\indent\sc Definition}
\newtheorem{remark}[theorem]{\indent\sc Remark}
\newtheorem{example}[theorem]{\indent\sc Example}

\begin{document}

\title[A large family of projectively equivalent $C^0$-Finsler manifolds]{A large family of projectively equivalent $C^0$-Finsler manifolds} 

\author[Ryuichi Fukuoka]{Ryuichi Fukuoka} 


%


\address{Department of Mathematics \endgraf 
State University of Maring\'a \endgraf
87020-900, Maring\'a, PR \endgraf
Brazil}
\email{rfukuoka@uem.br}



\maketitle

\begin{abstract}
A $C^0$-Finsler structure on a differentiable manifold is a continuous real valued function defined on its tangent bundle such that its restriction to each tangent space is a norm. 
In this work we present a large family of projectively  equivalent $C^0$-Finsler manifolds $(\hat M,\hat F)$, where $\hat M$ is diffeomorphic to the Euclidean plane. The structures $\hat F$ don't have partial derivatives and they aren't invariant by any transformation group of $\hat M$.
For every $p,q \in (\hat M,\hat F)$, we determine the unique minimizing path connecting $p$ and $q$.
They are line segments parallel to the vectors $(\sqrt{3}/2,1/2)$, $(0,1)$ or $(-\sqrt{3}/2,1/2)$, or else a concatenation of two of these line segments.
Moreover $(\hat M,\hat F)$ aren't Busemann $G$-spaces and they don't admit any bounded open $\hat F$-strongly convex subsets.
Other geodesic properties of $(\hat M,\hat F)$ are also studied.
\end{abstract}

\section{Introduction}
\label{introducao}

Let $M$ be a differentiable manifold, $T_xM$ be its tangent space at $x \in M$ and $TM=\{(x,y); x\in M, y \in T_xM\}$ be its tangent bundle.
$TM$ without the zero section is the {\em slit tangent bundle} and it will be denoted by $TM\backslash \{0\}$.
If $\phi=(x^1, \ldots, x^n)$ is a coordinate system on an open subset $U$ of $M$, then
\[
\left( \phi^{-1}(x^1, \ldots, x^n), \sum_{i=1}^n y^i\frac{\partial}{\partial x^i} \right) \mapsto (x^1, \ldots, x^n, y^1, \ldots, y^n)
\]
is the {\em natural coordinate system} on $TU$ with respect to $\phi$.
We use the superscript notation for coordinate functions as it is usual in Finsler geometry.

A Finsler structure on $M$ is a function $F: TM \rightarrow \mathbb R$ which is smooth when restricted to the slit tangent bundle and its restriction to each tangent space is a Minkowski norm (See \cite{BaoChernShen}).
There is another definition of Finsler structure on a differentiable manifold:
It is a continuous function $F:TM \rightarrow \mathbb R$ such that its restriction to each tangent space is a norm (See \cite{Burago}). 
In order to make distinction between these two objects, we call the latter by {\em $C^0$-Finsler structure}.

Riemannian geometry has been extremely successful in order to study geometry through  differential calculus.
Locally, Riemannian manifolds have some similarities with Euclidean spaces (for instance the existence of strongly convex geodesic balls) and it also provides geometrical invariants that tell us when two Riemannian manifolds can't be locally isometric (for instance, the curvature tensor).
Riemannian metrics also interact very well with certain topological objects. 
The Gauss-Bonnet theorem, the Hadamard theorem, the Bonnet-Myers theorem and the sphere theorem provide classical examples of this feature. 
All these results use strongly the possibility to differentiate the metric tensor.

Finsler geometry is a very relevant subject of differential geometry nowadays and it has the differential calculus as one of its main tools as well.
Its development has followed the footsteps of Riemannian geometry in many aspects and there are several similarities between these two theories (see \cite{BaoChernShen}).
But there are some differences as well.
For instance, Finsler manifolds which aren't Riemannian don't admit a canonical connection and a canonical volume form (see \cite{BaoChernShen}, \cite{Duran-volume}).

$C^0$-Finsler geometry is much less developed than Finsler geometry because differential calculus can't be applied directly on $C^0$-Finsler structures.
Moreover it doesn't have a model geometry such as Riemannian geometry to follow.
There are several differences between geodesics in Finsler manifolds and in $C^0$-Finsler manifolds.
For instance, for Finsler manifolds, we have existence and uniqueness of a geodesic with a given initial position and velocity.
In addition every geodesic is smooth.
This property doesn't hold for the family of $C^0$-Finsler manifolds that we introduce in this work. 
The lack of differentiability and the lack of strong convexity of $\hat F$ explain these differences.
Another example where we have a non-standard behavior of geodesics is the plane endowed with the maximum norm, which can be identified naturally with a $C^0$-Finsler manifold.
In this case, every point $p \in \hat M$ that doesn't lie in the lines $x^2=x^1$ or $x^2=-x^1$ has infinitely many minimizing paths connecting it to $(0,0)$ and several of them aren't differentiable (compare with Proposition \ref{mfo}).

$C^0$-Finsler structures appears naturally when we study intrinsic invariant metrics on homogeneous spaces.
Pioneering work in this direction are Berestovskii's papers \cite{Berestovskii1} and \cite{Berestovskii2}.
Let $M$ be a locally compact and locally contractible homogeneous space endowed with an invariant intrinsic metric $d_M$.
Berestovskii proved that $(M,d_M)$ is isometric to a left coset manifold $G/H$ of a Lie group $G$ by a compact subgroup $H < G$ endowed with a $G$-invariant $C^0$-Carnot-Carath\'{e}odory-Finsler metric.
He also proved that if every orbit of one-parameter subgroups of $G$ (under the natural action $\varphi: G \times G/H \rightarrow G/H$) is rectifiable, then $d_M$ is $C^0$-Finsler (see \cite{Berestovskii2}).
Here the $C^0$-Carnot-Carath\'{e}odory-Finsler metric comes from a completely nonholonomic $G$-invariant distribution $\mathcal D$  endowed with a $G$-invariant norm.  
The metric $d_M$ is defined analogously as in the Carnot-Carath\'{e}odory metric of sub-Riemannian geometry (see \cite{Montgomery}). 

Projectively flat metrics have been studied since the nineteenth century.
Beltrami proved in 1855 that if the geodesics of a Riemannian metric defined on an open subset of $\mathbb R^2$ are straight lines, then it has constant Gaussian curvature (See \cite{Beltrami}).
Hilbert Fourth Problem proposes the study of more general and not necessarily symmetric metrics on convex subsets of $\mathbb R^2$ such that the shortest paths are straight lines. 
Several solutions was given for this problem and they can be found in \cite{Busemann-Hilbert-4} and \cite{Pogorelov-Hilbert-4}. 
The history of projectively flat Finsler structures comes from the beginning of the twentieth century and it is a very active subject nowadays.
For instance an account about it as well as the study of the constant flag curvature case can be found in \cite{Bryant-projectively}, \cite{Li-advances},  \cite{Shen-transactions}.

Two Finsler structures on a differentiable manifold are projectively equivalent if they have the same geodesics as point sets.
In \cite{Levi-Civita}, Levi-Civita studied the local problem of projectively equivalent Riemannian manifolds.
Given a Riemannian manifold $(M,g)$, the family of Riemannian metrics which are projectively equivalent to $g$ can be represented by a solution of a linear differential equation (see \cite{Kiosak-Matveev} for a more modern mathematical language).
In particular, they can be represented as a finite dimensional vector space. 
The global study of projectively equivalent Riemannian manifolds is related to integrable systems and frequently involves topological properties of manifolds (see for instance \cite{Matveev-torus}, \cite{Matveev-same-geodesics}).
For Finsler manifolds, a family of projectively equivalent Finsler structures doesn't have, in general, finite dimension.
For instance every vector space endowed with a Minkowski norm is projectively equivalent to the Euclidean space.

For $C^0$-Finsler manifolds, we don't have any systematic tool in order to study geodesics and projective equivalence, although in some particular cases it is possible do calculate geodesics (see \cite{Gribanova} and Section \ref{Secao-Gribanova}).


In this work we introduce a large family of projectively equivalent $C^0$-Finsler structures $\hat F$ in $\hat M \cong\mathbb R^2$.
$\hat F$ are of Berwald type, that is, all tangent spaces of $\hat M$ seen as normed spaces are pairwise isometric.
The $C^0$-Finsler structure is given by
\[
\hat F(x^1, x^2, y^1, y^2) = f(x^1, x^2) F_0(y^1,y^2),
\]
where $F_0$ is the norm on the $(y^1,y^2)$ plane where the unit sphere is the regular hexagon with vertices 
\[
\pm (0,1), \pm (\sqrt{3}/2,1/2) \text{ and } \pm (-\sqrt{3}/2,1/2)
\]
and $f$ is a positive continuous function.
For the sake of simplicity, we denote the $C^0$-Finsler structure
\[
(x^1, x^2, y^1, y^2) \mapsto F_0(y^1, y^2)
\]
on $\hat M$ also by $F_0$.
As in the Poincar\'e half-plane model of hyperbolic plane, $f(x^1, x^2)$ has the tendence to be smaller for larger $x^2$ (see (\ref{limite inferior deformacao f})).
For every $p,q \in (\hat M,\hat F)$, we calculate the unique minimizing path connecting them.
They are line segments parallel to the vectors $(\sqrt{3}/2,1/2)$, $(0,1)$ or $(-\sqrt{3}/2,1/2)$, or else a concatenation of two of these line segments.
Therefore for ``almost every pair of points'' $p$ and $q$, the minimizing path connecting them isn't differentiable.

These $C^0$-Finsler structures don't admit (a priori) partial derivatives and they aren't invariant by any group of transformations of $\hat M$.
Therefore we can't calculate geodesics explicitly as solutions of ODE's or else using the Pontryagin's maximum principle (see Section \ref{Secao-Gribanova}).
In order to calculate the minimizing paths of $(\hat M,\hat F)$, we compare them with minimizing paths of $(\hat M,F_0)$ (see Proposition \ref{mfo}) and use the fact that ``higher'' paths are shorter than ``lower'' paths.
Most of the proofs comes from metric geometry.

The contributions of the present work to the study of minimizing paths and projective equivalence in $C^0$-Finsler geometry are the following: 
\begin{enumerate}
\item There are $C^0$-Finsler manifolds where the minimizing paths can be calculated explicitly through comparison with a model space;
\item There are $C^0$-Finsler structures that are geodesically stable with respect to a large family of metric defomations, that is, the deformations are projectively equivalent to the original manifold. 
In contrast to the Finsler case, these deformations don't need to satisfy any differential equation (See Example \ref{exemplos c0 finsler}).
\end{enumerate} 

Concatenations of line segments parallel to the vectors $(\sqrt{3}/2,1/2)$, $(-\sqrt{3}/2,1/2)$ or $(0,1)$ are very important in this work and they are called {\em preferred paths}.
For technical reasons, the trivial line segment (a point) will be also considered parallel to these vectors and they can be part of a preferred path.
We also use the term ``preferred'' for half-lines, lines and vectors which are parallel to these directions.

This work is organized as follows:
In Section \ref{define familia} we define the family of $C^0$-Finsler structures that we study in this work.
In addition we fix the ``clock'' notation for line segments in $\mathbb R^2$ in order to make the proofs easier to follow.
In Section \ref{Secao-Gribanova} we present a Gribanova's work that is related to this work.
In Section \ref{preliminares} we give definitions and theorems that are necessary for this work.
In Section \ref{preferred paths} we present several length comparison results between preferred paths.
In Section \ref{s-minimizing} we calculate explicitly the minimizing path {\em among preferred paths} that connects two arbitrary points.
In Section \ref{smooth and preferred}, we prove that if a piecewise smooth curve $\gamma$ connecting $p, q \in \hat M$ has a non-preferred tangent vector, then there exist a preferred path connecting $p$ and $q$ which is strictly shorter than $\gamma$.
This result determines all minimizing paths and geodesics in $(\hat M, \hat F)$ (See Theorems \ref{principal} and \ref{principal 2}).
In Section \ref{geodesic structure}, we prove that the manifolds $(\hat M, \hat F)$ don't admit any bounded open $\hat F$-strongly convex subsets. 
Moreover $(\hat M, \hat F)$ aren't Busemann $G$-spaces. 
We also make comments about other geodesic properties of $(\hat M,\hat F)$ and propose some problems.

The author would like to thank Hugo Murilo Rodrigues for his valuable suggestions.

\section{The family of $C^0$-Finsler structures on $\mathbb R^2$}
\label{define familia}

In what follows, $\hat M = \mathbb R^2$ is the differentiable manifold endowed with its canonical coordinate system $(x^1, x^2)$.
The tangent bundle
$T\hat M$ is endowed with its natural coordinates $(x^1, x^2, y^1, y^2)$. 

The vector $\vec v_i$ in $\mathbb R^2$, with $i \in \{ 0, \ldots, 11 \}$, denotes the Euclidean unit vector which has the same direction of the hour hand of a clock when it is $i$-o'clock.
For instance, $\vec v_0=(0,1)$ and $\vec v_3=(1,0)$.
We use this notation because it is more intuitive and simpler in this work than the traditional angle notation.

We denote the Euclidean oriented closed line segment in $\mathbb R^2$ connecting $p_1$ and $p_2$ by $[p_1,p_2]$.
The notation
$[p_1, p_2, \ldots, p_n]$ is used for the concatenation of the Euclidean line segments $[p_1, p_2]$, $[p_2, p_3], \ldots, [p_{n-1}, p_n]$. 
Whenever applicable, we use the clock convention for concatenation of line segments: 
When we state that $[p_1, p_2, \ldots, p_n]$ has directions $\left< \alpha_1, \ldots, \alpha_{n-1}\right>$, with $\alpha_1, \ldots, \alpha_{n-1} \in \{0, \ldots, 11\}$, it means that the vector $\overrightarrow{p_ip_{i+1}}$ is a positive multiple of $\vec v_i$. 

The half-line beginning from a point $p$ with the same direction of $\vec v_i$ is denoted by $h_i(p)$ and the line containing this half-line is denoted by $l_i(p)$.
If the point $p$ isn't important in $h_i(p)$ and $l_i(p)$ (for instance, when we are interested if some object is orthogonal to them), we replace $h_i(p)$ and $l_i(p)$ respectively by $h_i$ and $l_i$.
The Euclidean line containing two distinct points $p, q \in \hat M$ is denoted by $l[p,q]$.
The Euclidean half-line beginning at $a$ and containing $p \neq a$ is denoted by $h[a,p]$.

Now we are going towards the definition of the family of $C^0$-Finsler structures $\hat F$ on $\hat M$.
$\hat F: T\hat M \rightarrow \mathbb R$ is defined by 
\begin{equation}
\label{define F}
\hat F(x^1, x^2, y^1, y^2) = f(x^1, x^2) F_0 (x^1, x^2, y^1, y^2) 
\end{equation}
where $f: \hat M \rightarrow \mathbb R$ is a continuous positive function such that:
\begin{itemize}
\item \label{item 1 f} There exist a $\hat \theta \in (\pi/3, \pi/2]$ such that if $p \neq q$ and the Euclidean angle between $\overrightarrow{pq}$ and $\vec v_0$ is in the interval $[-\hat \theta, \hat \theta]$, then
\begin{equation}
\label{limite inferior deformacao f}
f(q)< f(p);
\end{equation}
\item \begin{equation}
\label{limite superior deformacao f}
f(x) \in (1, L_{\hat \theta}) \text{ for every }x \in \hat M,
\end{equation}
where $L_{\hat \theta} > 1$ depends only on $\hat \theta$ and will be explained after Remark \ref{semi plano e local}.
\end{itemize}

Denote by $A(p)$ the set of points $q \in \hat M\backslash \{p\}$ such that the angle between $\vec v_0$ and $\overrightarrow{pq}$ is in the interval $[-\hat \theta, \hat \theta]$. 
$\vec v_{\hat \theta}$ is the unit Euclidean vector such that the angle between $\vec v_0$ and $\vec v_{\hat \theta}$ (in this order) is $\hat \theta$. 
The vector $\vec v_{-\hat \theta}$ is defined analogously.
Although this notation doesn't follow exactly the ``clock pattern'' presented before, this notation will not cause confusion because there isn't any integer in the interval $(\pi/3,\pi/2]$.  
The subscript $\pm \hat\theta$ is also used for Euclidean line segments, lines and half-lines. 

If $f$ is differentiable, then (\ref{limite inferior deformacao f}) implies that the angle between $\vec v_6$ and $\nabla f$ is in the interval $[-(\pi/2 - \hat \theta), \pi/2 - \hat \theta]$ whenever $\nabla f$ different from zero.

\begin{remark}
\label{semi plano e local}
In this work we prove all the results for $\hat M = \mathbb R^2$.
But afterwards it will become clear that we can replace $\hat M$ by ``$\hat F$-convex'' open subsets $U$ of $\hat M$, which includes some arbitrarily small neighborhoods of a point (see Proposition \ref{local}).
\end{remark}

Now we define $L_{\hat \theta}$.
Consider the Euclidean trapezoid $T = [a, b, c, d, a] \subset \hat M$ with directions $\left< 8, 6, 2, 10 \right>$ such that $h_{-\hat \theta}(b)$ contains $d$.
It is easy to see that this trapezoid exist and that trapezoids with these properties are pairwise Euclidean homothetic  ({\em here and in several parts of this work, an Euclidean drawing is very helpful and enough to figure out the situation}).
We want to make sure that 
\begin{equation}
\label{desigualdade fundamental}
\ell_{\hat F} ([c,d]) < \ell_{\hat F} ([d,a,b,c]),
\end{equation}
where $\ell_{\hat F}$ is the length with respect to the $C^0$-Finsler structure $\hat F$.
Of course this inequality holds if $\hat F$ is replaced by $F_0$.
We claim that if we define
\begin{equation}
\label{l theta}
L_{\hat \theta} = \frac{\ell_E ([d,a,b,c])}{\ell_E([c,d])} > 1,
\end{equation}
where $\ell_E$ is the Euclidean length, then (\ref{desigualdade fundamental}) holds.
In fact, first of all observe that
\[
\frac{\ell_{F_0}([a,b])}{\ell_{E}([a,b])} = \frac{\ell_{F_0}([b,c])}{\ell_{E}([b,c])} 
= \frac{\ell_{F_0}([c,d])}{\ell_{E}([c,d])} = 
\frac{\ell_{F_0}([d,a])}{\ell_{E}([d,a])} = 1
\]
due to definition of $[a,b]$, $[b,c]$, $[c,d]$, $[d,a]$ and $F_0$.
Therefore
\begin{equation}
\label{reference trapezoid}
\ell_{\hat F} ([c,d])< L_{\hat \theta} . \ell_E ([c,d]) =\ell_E([d,a,b,c]) < \ell_{\hat F} ([d,a,b,c]).
\end{equation}

It is straightforward from the definition of $L_{\hat \theta}$ that $L_{\hat \theta} \in (1, 2)$.

Now we present a large family of $C^0$-Finsler manifolds that satisfy (\ref{limite inferior deformacao f}) and (\ref{limite superior deformacao f}).
It will be proved afterwards that they are projectively equivalent (see Theorem \ref{principal 2}).

\begin{example}
\label{exemplos c0 finsler} 
Let $\hat \theta \in (\pi/3, \pi/2)$ and $L_{\hat \theta}$ given by (\ref{l theta}). 
Let $f_1: \mathbb R \rightarrow \mathbb R$ be a Lipschitz function with Lispchitz constant $L \in (0, \arctan (\pi/2 - \hat \theta))$.
Let  $f_2: \mathbb R \rightarrow \left( 1, L_{\hat \theta}\right)$ be a continuous strictly decreasing function.
Then it is straightforward that
\[
f(x^1,x^2) = f_2(x^2+f_1(x^1))
\]
satisfies (\ref{limite superior deformacao f}).
We claim that $f$ satisfies (\ref{limite inferior deformacao f}).
In fact, suppose that $q=(q^1, q^2)$ and $p=(p^1, p^2)$ are  such that $q\in A(p)$.
Then
\[
f_1(p^1)-f_1(q^1) \leq L\vert q^1 - p^1 \vert < q^2 - p^2,
\]
where the second inequality holds because $q \in A(p)$.
Therefore 
\[
f_1(p^1) + p^2 < f_1(q^1) + q^2
\]
and 
\[
f(p) > f(q)
\]
whenever $q \in A(p)$.
\end{example}

We will work with three $C^0$-Finsler structures on $\hat M$: $\hat F$, $F_0$ and the Euclidean norm $E$.
Whenever necessary, we will make these structures explicit. 
For instance (as already defined) the arclength with respect to $\hat F$, $F_0$ and $E$ is denoted respectively by $\ell_{\hat F}$, $\ell_{F_0}$ and $\ell_E$.
{\em If there isn't any mention to the $C^0$-Finsler structure used on $\hat M$, it will be assumed implicitly that the objects are measured with respect to $\hat F$.} 
Any reference to angles is with respect to $E$.

\section{Gribanova's work}

\label{Secao-Gribanova}

In \cite{Gribanova}, Gribanova studied left invariant $C^0$-Finsler metrics $F$ on 
\[
\mathbb R^2_+ = \{(x^1,x^2)\in \mathbb R^2;x^2 > 0\}
\]
endowed with the group of transformations generated by homotheties centered at the origin and horizontal translations.
The left invariance of $F$ implies that 
\[
F(x^1,x^2,y^1,y^2) = \frac{F(0,1,y^1,y^2)}{x^2}
\]
(as it happens in hyperbolic plane), 
$F$ is continuously differentiable with respect to $(x^1,x^2)$ and Pontryagin's maximum principle can be used in order to get a ``geodesic type equation'' for this problem. 
This equation gives a necessary condition in order to a path $\gamma:[a,b] \rightarrow \mathbb R^2_+$ be minimizing. 
Gribanova calculated all the minimizing paths of these examples in Theorem 1 of \cite{Gribanova}.
For the particular $C^0$-Finsler structure
\[
F(x^1, x^2, y^1, y^2) = \frac{F_0(y^1, y^2)}{x^2},
\]
which is related to the present work, the minimizing paths are given by 
\[
\{(x^1, x^2) \in \mathbb R^2_+;F_0(x^1 - p^1 ,x^2) = R\},
\] 
where $R>0$ and $p^1 \in \mathbb R$.
In particular, every preferred path $[a,b,c,d,e]$ with directions $\left< 0, 2, 4, 6\right>$ is a geodesic.
Therefore there exist geodesics that aren't minimizing paths.
This is a slight oversight in Theorem 1 of \cite{Gribanova}, because she states that every geodesic is a minimizing path.

\section{Preliminaries}
\label{preliminares}

In this section we present notations, definitions and results that are necessary for this work. 
A reference for $C^0$-Finsler manifolds is \cite{Burago}.
The definition of Busemann $G$-space can be found in \cite{Busemann-Geodesic}. 

If $X$ is a topological space and $U$ is a subset of $X$, then $\text{int} U$ is the interior of $U$, $\bar U$ is the closure of $U$ and $\partial U$ is the boundary of $U$.

Let $(M, F)$ be a $C^0$-Finsler manifold. 
We denote the length of a piecewise smooth curve $\gamma:[t_0,s_0]\rightarrow (M,  F)$ by
\begin{equation}
\label{comprimento finsler}
\ell_F(\gamma):=\int_{t_0}^{s_0} F(\gamma(t),\gamma^\prime(t))dt.
\end{equation}
The metric (distance function) $d_F: M \times M \rightarrow \mathbb R$ on $(M,F)$ is given by
\[
d_F(p,q)=\inf_{\gamma \in \mathcal S_{p,q}} \ell_F (\gamma),
\]
where $\mathcal S_{p,q}$ is the family of piecewise smooth paths on $M$ that connects $p$ and $q$. 
It is straightforward that $(M, F)$ has the same topology of the differentiable manifold $M$.

Given a metric space $(X,d)$ and a path $\gamma:[t_0,s_0] \rightarrow X$, the length of $\gamma$ is defined as
\[
\ell_d(\gamma):=\sup_{\mathcal P}\sum_{i=1}^{n_{\mathcal P}} d(\gamma(\tau_i),\gamma(\tau_{i-1})),
\]
where $\mathcal P=\{t_0 = \tau_0 < \tau_1 < \ldots < \tau_{n_{\mathcal P}}=s_0\}$ is a partition of $[t_0,s_0]$. The supremum is taken over all partitions of $[t_0,s_0]$.

A natural question is whether 
\[
\ell_{F} (\gamma) = \ell_{d_F} (\gamma)
\]
holds for every piecewise smooth path $\gamma:[t_0,s_0] \rightarrow M$. 
The answer is affirmative (see Section 2.4.2 of \cite{Burago}).

\begin{remark}
\label{equivalent metrics}
It is well known that the metrics $d_E$ and $d_{F_0}$ are equivalent.
The equivalence between $d_{F_0}$ and $d_{\hat F}$ follows from (\ref{limite superior deformacao f}).
Therefore there exist a constant $C>0$ such that
\[
\ell_E \leq \ell_{F_0} < \ell_F \leq C \ell_E.
\]
\end{remark}

A path $\gamma:I\subset \mathbb R \rightarrow X$ is a geodesic if it is a locally minimizing path, that is for every $t_0\in I$, there exist a neighborhood $J$ of $t_0$ such that $\gamma\vert_{[t_1,t_2]}$ is a minimizing path for every $t_1\leq t_2 \in J$.

\begin{remark}
\label{caminhos e caminhos parametrizados} 
Two parameterized paths $\gamma_1$ and $\gamma_2$ will be identified if they  differ by a monotonic reparameterization and we will use the relationship $\gamma_1 = \gamma_2$ for the sake of simplicity.
This convention allow us to identify parameterized paths $\gamma$ with its image if $\gamma$ is injective. 
Whenever there exist possibility of misunderstandings with these identifications, we will provide further explanations to make the situation clearer.
\end{remark}

\begin{definition}
\label{Busemann definition}[Busemann $G$-space]
A Busemann $G$-space is a metric space $(X,d_X)$ satisfying the following conditions:
\begin{enumerate}
\item A bounded subset of $(X,d_X)$ with infinite points has an accumulation point;
\item If $a, b\in X$ are two distinct points, there exist $c \in X\backslash \{a,b\}$ such that $d_X(a,c) + d_X(c,b) = d_X(a,b)$;
\item For every $p \in X$, there exist a $\rho > 0$ such that if $a,b \in X$ are distinct points satisfying $d_X(p,a) < \rho$ and $d_X(p,b)< \rho$, then there exist $c \in X\backslash \{a,b\}$ such that $d_X(a,b)+d_X(b,c) = d_X(a,c)$;
\item If $a,b\in X$ are distinct points and $c_1, c_2 \in X\backslash \{a,b\}$ are such that 
\begin{itemize}
\item $d_X(a,b)+d_X(b,c_i)=d_X(a,c_i)$, for $i=1,2$;
\item $d_X(b,c_1) = d_X(b,c_2)$, 
\end{itemize}
then $c_1 = c_2$.
\end{enumerate}
\end{definition}

For instance, complete Riemannian manifolds are Busemann $G$-spaces. In fact, Items (1) and (2) are due to the Hopf-Rinow theorem, Item (3) is due to the existence of strongly convex geodesic balls and Item (4) is consequence of the fact that geodesics aren't minimizing beyond their cut points (See \cite{doCarmo2}).

Now we define some geometrical objects in $\hat M$.

We denote by $S_k(p)$, $k \in \{1, 3, 5, 7, 9, 11\}$ the open sector of $\hat M$ with angle $\pi/3$ bounded by the preferred half-lines $h_{k-1}(p)$ and $h_{k+1}(p)$.
Of course we identify $h_0(p)$ with $h_{12}(p)$.
Whenever the point $p\in \hat M$ isn't relevant, we write $S_k$ instead of $S_k(p)$.

Let $[p_1,\ldots, p_n]$ be a preferred path. 
The points $p_1, \ldots, p_n$ are the {\em vertices} of $[p_1,\ldots, p_n]$.
$p_{i-1}$ is the {\em predecessor}, $p_{i+1}$ is the {\em successor} of $p_i$ and we denote them respectively by $p(p_i)$ and $s(p_i)$.
If $p_1 = p_n$, then $p_2$ is the successor of $p_1=p_n$ and $p_{n-1}$ is its predecessor.
A vertex $p_i$ is called effective if one of the following items holds:
\begin{enumerate}
\item it doesn't have a predecessor and a successor;
\item it only has a successor and $s(p_i)\neq p_i$; 
\item it only has a predecessor and $p_i\neq p(p_i)$;
\item it has a successor and a predecessor, $s(p_i) \neq p_i$,  $p_i \neq p(p_i)$ and the angle between $s(p_i)-p_i$ and $p_i - p(p_i)$ is different from $0$.
\end{enumerate} 
A path $[p_1, \ldots, p_n]$ is called {\em simplified} if it has only effective vertices.
{\em The set of preferred paths will be denoted by $\mathbf P$.
The set of simplified preferred paths will be denoted by $\mathbf S$.} 
A simplified preferred path can be represented as $[p_1, \ldots, p_n]$ with directions $\left< \alpha_1, \ldots, \alpha_{n-1} \right>$, where $\alpha_i \in \{0, 2, 4, 6, 8, 10\}$ for every $i\in \{1, \ldots, n-1\}$ and $\alpha_i \neq \alpha_{i+1}$ for every $i \in \{1, \ldots, n-2\}$.

\begin{remark}
\label{simplied path}
We can get a simplified path, equivalent to a preferred path $[p_1, \ldots,$ $ p_n]$, in the following way: 
Drop all the vertices $p_i$ such that $p_i=p(p_i)$. 
After that, drop all the vertices that doesn't satisfy Item (4) above.
Then we get a simplified path $\gamma$ with the same trace and the same length (with respect to any metric) as $[p_1, \ldots, p_n]$.
We call $\gamma$ a simplification of $[p_1, \ldots, p_n]$.
\end{remark}

Let $p,q \in \hat M$ with $p\neq q$. 
Now we define $[p,q]_{\text{min}}$ which will be proved afterwards that it is the unique minimizing path connecting $p$ and $q$.
If $[p,q] \in \mathbf S$, then define $[x,y]_{\text{min}}=[x, y]$.
Otherwise there exist a unique angular sector $S_k(p)$ containing $q$.
Suppose that the upper boundary of $S_k(p)$ is $h_i(p)$ and its lower boundary is $h_j(p)$.
For instance, $S_7(p)$ has $h_8(p)$ as its upper boundary and $h_6(p)$ as its lower boundary. Denote $a=l_i(p) \cap l_j(q)$.
Then $[p,q]_{\text{min}}$ is defined as $[p,a,q]$.
It's not difficult to see that $[p,q]_{\text{min}} = [q,p]_{\text{min}}$ as subsets of $\hat M$.
Analysing all the cases we have that $[p,a,q]$ can have directions $\left< 0, 2\right>$, $\left< 2, 4\right>$ or else $\left< 4, 6\right>$ (or their reverse paths).
All these cases have concavity downwards.

Remark \ref{explica minimizante intuitivamente} explain intuitively why $[p,q]_{\text{min}}$ are the minimizing paths connecting two points in $(\hat M,\hat F)$ comparing them with the corresponding minimizing paths in $(\hat M, F_0)$.

\begin{proposition}[Minimizing paths of $(\hat M, F_0)$]
\label{mfo}
Let $p,q \in (\hat M,F_0)$. 
We have that
\begin{enumerate}
\item Euclidean line segments in $(\hat M,F_0)$ are minimizing paths; 
\item If $[p,q] \in \mathbf P$, then $[p,q]$ is the only piecewise smooth minimizing path  connecting $p$ and $q$;
\item If $q \in S_3(p)$, then the minimizing paths are monotonic reparameterization of paths $\gamma(t) = (t,x^2(t))$, where $x^2(t)$ are a Lipschitz map with Lipschitz constant equal to $\sqrt{3}/3$;
\item If $q \in S_k(p)$, $k\in \{1, 5, 7, 9, 11\}$, then the minimizing paths are the paths described in Item $(3)$ rotated by an angle $(3-k)\pi /6$, $k\in \mathbb Z$.
\end{enumerate} 

\end{proposition}

\begin{proof}

These facts are well known but we give their proofs for the sake of completeness. 

\

Item 1:

It is proved in Proposition 3.4 of \cite{BenettiFukuoka}.

\

Item 2:

Let $\gamma:[0,1] \rightarrow \mathbb (\hat M,F_0)$ be a path connecting $p$ and $q$ such that $\gamma([0,1]) \not \subset [p,q]$. 
Then 
\[
\ell_{F_0} ([p,q]) = \ell_E ([p,q]) < \ell_E (\gamma) \leq \ell_{F_0}(\gamma),
\]
where the last inequality holds due to $\ell_E \leq \ell_{F_0}$.

\

Item 3:

Observe that $\sqrt{3}/3$ and $-\sqrt{3}/3$ are the slopes of $h_2(p)$ and $h_4(p)$ respectively.

If $(z^1, z^2) \in \bar S_3((x^1,x^2))$, then
\[
d_{F_0}((x^1,x^2),(z^1,z^2)) 
= \frac{2}{\sqrt 3} \left( z^1 - x^1 \right)
\]
and if $z \not\in \bar S_3(x)$, then
\[
d_{F_0}((x^1,x^2),(z^1,z^2)) 
> \frac{2}{\sqrt 3} \left( z^1 - x^1 \right).
\]
In particular, 
\[
d_{F_0}(p,q)=\frac{2}{\sqrt 3}\left( q_1 - p_1\right).
\]

The length of a path $\eta = (\eta_1, \eta_2)$ connecting $p$ and $q$ is given by
\[
\ell_{F_0} (\eta) 
= \sup_{\mathcal P}\sum_{i=1}^{n_{\mathcal P}} d_{F_0} (\eta(t_i), \eta(t_{i+1}))
\geq \sup_{\mathcal P}\sum_{i=1}^{n_{\mathcal P}} \frac{2}{\sqrt 3}\left( \eta^1(t_{i+1}) - \eta^1(t_i)\right),
\]
and the equality holds iff $\eta (t) \in \bar S_3(\eta(s))$ for every $t>s$.
In the equality case we have that
\[
\ell_{F_0}(\eta) = \frac{2}{\sqrt 3} (q^1 - p^1) = d_{F_0} (p,q).
\]
But the condition $\eta(t) \in \bar S_3(\eta(s))$ for every $t>s$ is equivalent to the condition that $\eta$ can be reparameterized as $\gamma(t) = (t,x^2(t))$, where $x^2(t)$ is a Lipschitz map with Lipschitz constant $\sqrt 3/3$, what settles this item. 

\

Item (4)

Just observe that rotations centered at $p$ by an integer multiple of $\pi/3$ are isometries of $(\hat M, F_0)$ and they interchanges the subsets $S_k(p)$, $k\in \{1, 3, 5, 7, 9,$ $ 11\}$.
\end{proof}

\begin{remark}
\label{explica minimizante intuitivamente}
Items (3) and (4) of Proposition \ref{mfo} states that, in general, several paths that connects $p$ and $q$ are minimizing in $(\hat M,F_0)$.
When we introduce a pointwise homothety $f(x^1,x^2)$ on $F_0$ in such a way that higher paths of $(\hat M,\hat F)$ have the tendence to be shorter than the lower paths, it is a kind of tie break between these paths.
{\em The minimizing path connecting $p$ and $q$ in $(\hat M, \hat F)$ is $[p,q]_{\text{min}}$, which is the highest among all of these paths.}

In order to keep this comparison under control, (\ref{limite inferior deformacao f}) provides a type of lower bound for this pointwise deformation and (\ref{limite superior deformacao f}) provides a type of upper bound.
\end{remark} 

\section{Length comparisons between preferred paths}
\label{preferred paths}

In this section we give several length comparison results between preferred paths in $(\hat M,\hat F)$.

\begin{lemma}
\label{compara euclideano com F}
Let $\gamma_1$ and $\gamma_2$ be preferred paths such that $L_{\hat \theta}.\ell_E (\gamma_1) \leq \ell_E (\gamma_2)$.
Then $\ell_{\hat F} (\gamma_1) < \ell_{\hat F} (\gamma_2)$.
\end{lemma}

\begin{proof}

The proof is just straightforward calculation.

\[
\ell_{\hat F} (\gamma_1)  < L_{\hat \theta} \ell_{F_0} (\gamma_1) = L_{\hat \theta} \ell_{E} (\gamma_1) \leq \ell_E(\gamma_2) = \ell_{F_0}(\gamma_2) < \ell_{\hat F} (\gamma_2)
\]
due to the definition of $\hat F$ and the fact that $\gamma_1$ and $\gamma_2$ are preferred paths.

\end{proof}

\

Now we are going to prove that if $p, q \in \hat M$, then 
\[
\ell_{\hat F}([p,q]_{\text{min}}) \leq \ell_{\hat F}(\gamma)
\]
for every $\gamma \in \mathbf S$ connecting $p$ and $q$ which is placed below $[p,q]_{\text{min}}$.
We also prove that the equality holds iff $\gamma = [p,q]_{\text{min}}$ (see Propositions \ref{abaixo eh maior 1}, \ref{dos lados eh maior} and \ref{abaixo eh maior}).
We give some preliminary results before.

\begin{lemma}
\label{compara segmentos 0} Let $[a,b]$ and $[c,d]$ are line segments such that
\begin{enumerate}
\item $\ell_{F_0} ([a, b]) = \ell_{F_0} ([c, d])$;
\item $(1-t)c + t d \in A((1-t) a + t b)$ for every $t \in (0,1)$.
\end{enumerate} 
Then
\[
\ell_{\hat F} ([c, d]) < \ell_{\hat F} ([a, b]).
\]
\end{lemma}

\begin{proof}

Consider parameterizations $\gamma_a : [0,1] \rightarrow \hat M$ and $\gamma_c : [0,1] \rightarrow \hat M$ given respectively by $\gamma_a(t) = (1-t) a + t b$ and $\gamma_c (t) = (1-t) c + t d$.
Then
\[
\ell_{\hat F} ([c, d]) 
= \int_0^1 \hat F(\gamma_c(t), \gamma^\prime_c(t)).dt 
= \int_0^1 f(\gamma_c(t)) F_0 (d - c).dt
\]
\[
< \int_0^1 f(\gamma_a(t)) F_0 (b - a).dt = \ell_{\hat F} ([a, b]),
\]
where the inequality holds because $\gamma_c(t) \in A(\gamma_a(t))$ for every $t \in (0,1)$.

\end{proof}

\

Lemma \ref{compara segmentos 0} can be applied in several situations:

\begin{lemma}
\label{lados paralelos}
Suppose that $a,b,c,d \in \hat M$ are such that $\overrightarrow{ab} = \overrightarrow{cd}$ and $c \in A(a)$.
Then $\ell_{\hat F} ([c,d]) < \ell_{\hat F} ([a,b])$. 
\end{lemma}

\begin{proof}
The proof is analogous to the proof of Lemma \ref{compara segmentos 0}.
\end{proof}

\begin{lemma}
\label{compara segmentos 2} 
Let $[a, b] \in \mathbf S$ and suppose that $k\in \{0, 2, 4\}$ is such that $l_k(a) \neq l_k(b)$ $($that is, $[a,b]$ isn't parallel to $l_k$$)$.
Let $[c,d] \in \mathbf S$ such that
\begin{enumerate}
\item $c \in l_k(a)$ and $d \in l_k(b)$;
\item $c^2 \leq a^2$ and $d^2 \leq b^2$. 
\end{enumerate}
Then
\[
\ell_{\hat F}([c,d]) \geq \ell_{\hat F}([a,b])
\]
and the equality holds iff $[c,d]=[a,b]$.
\end{lemma}

\begin{proof}
In order to fix ideas, suppose that $k=2$.
The other cases are analogous.
The line segments $[a,b]$ and $[c,d]$ can have directions $\left< 0\right>$, $\left< 4\right>$, $\left< 6\right>$ or else $\left< 10\right>$.
Therefore $\ell_{F_0} ([b_i, b_{i+1}]) = \ell_{F_0}([c_i, c_{i+1}])$.

Suppose that $[c,d]$ is strictly below $[a,b]$. 
Then $(1-t)a + t b \in h_2((1-t)c + t d)$, what implies that $(1-t)a + t b \in A((1-t)c + t d)$ for every $t \in (0,1)$.
Therefore
\[
\ell_{\hat F}([c,d]) > \ell_{\hat F}([a,b])
\]
due to Lemma \ref{compara segmentos 0}.
\end{proof}

\begin{proposition}
\label{abaixo eh maior 1}
Let $[p,q] \in \mathbf S$ be a line segment parallel to $l_2$ or $l_4$.
If $\gamma \in \mathbf S$ connects $p$ and $q$ and its points are placed in the lower closed half-plane bounded by $l[p,q]$, then
\[
\ell_{\hat F}([p,q]) \leq \ell_{\hat F}(\gamma),
\]
and the equality holds iff $\gamma = [p,q]$. 
\end{proposition}

\begin{proof}

Denote $\gamma = [p=a_1, a_2, \ldots, q=a_n]$.
Consider $I=\{l_0(a_i) \cap [p,q];i\in \{1, \ldots, n\}\}$.
Reorder them as $I=\{p=b_1, b_2, \ldots, b_m=q\}$ in such a way that the points $b_i$ are placed from left to right.
The closed slab $\bar W_i$ bounded by $l_0(b_i)$ and $l_0(b_{i+1})$, $i \in \{1, \ldots, m-1\}$, intercepts $\gamma$.
This intersection can be decomposed in line segments (not necessarily disjoint) with endpoints in $l_0(b_i)$ and $l_0(b_{i+1})$ because the open slab $W_i$ doesn't intercept $\{a_i,i \in \{1, \ldots, n\}\}$.
Let $I_i$ be one of these line segments.
Then we have that
\[
\ell_{\hat F}([b_i,b_{i+1}]) \leq \ell_{\hat F}(I_i)
\]
due to Lemma \ref{compara segmentos 2} and because $I_i$ isn't placed above $[b_i,b_{i+1}]$.
The equality holds iff $\bar W_i \cap \gamma = [b_i, b_{i+1}]$.

The proposition is settled extending the analysis for every $i \in \{1, \ldots,$ $ m-1\}$.
\end{proof}

\begin{proposition}
\label{dos lados eh maior}
Let $[p,q] \in \mathbf S$ be a line segment parallel to $l_0$.
If $\gamma \in \mathbf S$ connects $p$ and $q$, then
\[
\ell_{\hat F}([p,q]) \leq \ell_{\hat F}(\gamma),
\]
and the equality holds iff $\gamma = [p,q]$. 
\end{proposition}

\begin{proof}

The line $l[p,q]$ is the boundary of two closed half-planes: $\bar H_L$ and $\bar H_R$, which are placed respectively on the left hand side and on the right hand side of $l[p,q]$.

Denote $\gamma = [p=a_1, a_2, \ldots, q=a_n]$.

Suppose that $\gamma$ is placed in $\bar H_R$.
Consider $I=\{l_4(a_i) \cap [p,q];i\in \{1, \ldots, n\}\}$.
Reorder them as $I=\{p=b_1, b_2, \ldots, b_m=q\}$ in such a way that the points $b_i$ are placed from the top to the bottom. 
The closed slab $\bar W_i$ bounded by $l_4(b_i)$ and $l_4(b_{i+1})$ intercepts $\gamma$.
This intersection can be decomposed in line segments (not necessarily disjoint) with endpoints in $l_4(b_i)$ and $l_4(b_{i+1})$ because the open slab $W_i$ doesn't intercept $\{a_i,i \in \{1, \ldots, n\}\}$.
Let $I_i$ be one of these line segments.
Then
\[
\ell_{\hat F}([b_i,b_{i+1}]) \leq \ell_{\hat F}(I_i)
\]
due to Lemma \ref{compara segmentos 2} because $I_i$ is not placed above $[b_i,b_{i+1}]$.
The equality holds iff $\bar W_i \cap \gamma = [b_i,b_{i+1}]$.
The proposition is settled for this case extending the analysis for every $i \in \{1, \ldots, n-1\}$.

If $\gamma$ is placed in $\bar H_L$, then use the direction $\left< 2 \right>$ instead of $\left< 4 \right>$ and the conclusion follows likewise.

If part of $\gamma$ is placed in $\bar H_L$ and the other part is placed in $\bar H_R$, then define 
\[
I = \{ l_4(a_i) \cap [p,q]; a_i \in \bar H_R, i=1, \ldots, n\} 
\]
\[
 \cup \{ l_2(a_i) \cap [p,q]; a_i \in \bar H_L, i=1, \ldots, n\}. 
\]
Reorder them as $I=\{p=b_1, b_2, \ldots, b_m=q\}$ in such a way that the points $b_i$ are placed from the top to the bottom. 
Now $[b_i,b_{i+1}]$ is compared with a line segment $I_i \subset \gamma$ placed in $\bar H_L$ which lies in the  closed slab bounded by $l_2 (b_i)$ and $l_2(b_{i+1})$ or else to a line segment $I_i \subset \gamma$ placed in $\bar H_R$ which lies in the  closed slab bounded by $l_4 (b_i)$ and $l_4(b_{i+1})$. 
The rest of the proof follows as in the former cases.
\end{proof}

The {\em concave side} of $[p,a,q] \in \mathbf S$ is the closed subset of $\hat M$ bounded by $h[a,p] \cup h[a,q]$ and containing $[p,q]$.
Notice that the concave side of $[p,a,q]$ is $h[a,p]$ whenever $h[a,p] = h[a, q]$.

\begin{proposition}
\label{abaixo eh maior}
Let $p, q \in \hat M$ distinct points such that $[p,q] \not \in \mathbf S$.
Suppose that $\gamma \in \mathbf S$ connects $p$ and $q$ and it lies on the concave side of $[p,q]_{\text{\em min}}$.
Then
\[
\ell_{\hat F}([p,q]_{\text{\em min}}) \leq \ell_{\hat F}(\gamma)
\]
and the equality holds iff $\gamma = [p,q]_{\text{\em min}}$.
\end{proposition}

\begin{proof}

The proof is almost identical to the proof of Proposition \ref{abaixo eh maior 1}.
For instance, if $[p,q]_{\text{min}} = [p,a,q]$ have directions $\left< 4, 6\right>$ and $\gamma = [a_1, \ldots, a_n]$, we consider 
\[
I=\{l_{\mathbf 2}(a_i) \cap [p,q];i\in \{1, \ldots, n\}\}.
\]
If $[p,q]_{\text{min}} = [p,a,q]$ have directions $\left< 2, 4\right>$, we consider 
\[
I=\{l_{\mathbf 0}(a_i) \cap [p,q];i\in \{1, \ldots, n\}\}.
\]
Finally if $[p,q]_{\text{min}} = [p,a,q]$ have directions $\left< 0, 2\right>$, we consider 
\[
I=\{l_{\mathbf 4}(a_i) \cap [p,q];i\in \{1, \ldots, n\}\}.
\]
The rest of the proof is analogous to the proof of Proposition \ref{abaixo eh maior 1}.
\end{proof}

The results above are enough to prove that preferred paths that connects $p$ and $q$ and are ``placed below'' $[p,q]_{\text{min}}$ are larger than $[p,q]_{\text{min}}$.
Most of the rest of this section present results which deal with preferred paths above $[p,q]_{\text{min}}$.

\begin{proposition}[Triangle sides comparison]
\label{triangulo equilatero}
Let $[a,b,c,a]$ be a triangle in $\mathbf S$ with directions $\left< 6,2,10 \right>$ or $\left< 6,10,2\right>$.
Then
\begin{enumerate}
\item $\ell_{\hat F} ([a,c]) < \ell_{\hat F} ([a, b]) < \ell_{\hat F} ([b,c])$;
\item $\ell_{\hat F} ([b,c])<\ell_{\hat F} ([a,c])+\ell_{\hat F}([a,b])$.
\end{enumerate}
\end{proposition}

\begin{proof}

Observe that preferred triangles in $\mathbf S$ are $E$-equilateral.

In order to prove that $\ell_{\hat F}([a,c]) < \ell_{\hat F}([a,b])$, observe that $l_2(b) = l_2(c)$, $l_2(a) \neq l_2(b)$ and that the line segments $[a,c]$ and $[a,b]$ satisfy the conditions of Lemma \ref{compara segmentos 2}.
Therefore
\[
\ell_{\hat F}([a,c]) < \ell_{\hat F}([a,b]).
\]
The proof of the relationship
\[
\ell_{\hat F}([a,b]) < \ell_{\hat F}([b,c])
\]
is analogous. 

The second item follows from Lemma \ref{compara euclideano com F} and the fact that 
\[
L_{\hat \theta}. \ell_E ([b, c]) < 2\ell_E ([b,c]) = \ell_E ([a,c]) + \ell_E ([a,b]).
\]
\end{proof}

The next result is a particular case of Proposition \ref{abaixo eh maior}.

\begin{corollary}[Parallelogram sides comparison 1]
\label{paralelogramo}
Let $P=[a,b,c,d,a]$ be a parallelogram in $\mathbf S$. 
\begin{enumerate}
\item  If $P$ has directions $\left< 8, 4, 2, 10\right> $, then 
\[
\ell_{\hat F} ([d,a,b]) < \ell_{\hat F} ([b, c, d]);
\]
\item  If $P$ has directions $\left< 6, 4, 0, 10\right>$ or $\left< 6, 8, 0, 2\right>$, then 
\[
\ell_{\hat F} ([a,d,c]) < \ell_{\hat F} ([a,b,c]).
\] 
\end{enumerate}
\end{corollary}

Now we prove that if $Q$ is a preferred trapezoid or a preferred parallelogram, then the length of one of its sides is less than the sum of the lengths of the other sides.
The final result is stated in Theorem \ref{um menor do que tres}. 

\begin{lemma}[Trapezoid sides comparison 1]
\label{trapezio1}
Let $T$ be a trapezoid in $\mathbf S$.
Then the length of one non-parallel side is less than the sum of the lengths of the other sides.
\end{lemma}

\begin{proof}

Denote the non-parallel side by $\sigma_1$ and the other sides by $\sigma_2$. 
Lemma \ref{compara euclideano com F} implies that
\[
\ell_{\hat F}(\sigma_1) < \ell_{\hat F}(\sigma_2)
\]
because $L_{\hat \theta}. \ell_E(\sigma_1) \leq \ell_E (\sigma_2)$.\end{proof}

\begin{lemma}[Parallelogram sides comparison 2]
\label{paralelogramo2}
Let $P$ be a parallelogram in $\mathbf S$.
Then the length of one $E$-shorter side $($or any side if $P$ is a diamond$)$ is less than the sum of the lengths of the other sides.
\end{lemma}

\begin{proof}
The proof is identical to the proof of Lemma \ref{trapezio1}.
\end{proof}

\begin{lemma}[Trapezoid sides comparison 2]
\label{trapezio2}
Let $T = [a,b,c,d,a]$ be a trapezoid in $\mathbf S$ with directions $\left< 8, 6, 2, 10\right>$ or $\left< 4, 6, 10, 2 \right>$.
Then the length of one side is less than the sum of the lengths of the other sides.
\end{lemma}

\begin{proof}

We analyse the case $\left< 8, 6, 2, 10\right>$.
The other case is vertically symmetric and its analysis is analogous.
We have only to prove that $\ell_{\hat F} ([c,d]) < \ell_{\hat F} ([d, a, b, c])$.
The inequality $\ell_{\hat F} ([a,b]) < \ell_{\hat F} ([b, c, d, a])$ is direct consequence of Proposition \ref{abaixo eh maior 1}. 
The other two cases are settled in Lemma \ref{trapezio1}.

We split the proof in several cases, depending on the placement of $d$ with respect to $l_{-\hat \theta}(b)$.

\

Case 1: $d \in l_{-\hat \theta}(b)$.

This case is exactly Inequality (\ref{reference trapezoid}).

\

Case 2: $d$ is below $l_{-\hat \theta}(b)$ (that is, compared to Case 1, the length of the parallel sides are relatively $E$-shorter than the nonparallel sides).

We have that
\[
L_{\hat \theta}.\ell_E ([c, d]) < \ell_E ([d, a, b, c])
\]
and
\[
\ell_{\hat F} ([c,d]) < \ell_{\hat F} ([d, a, b, c])
\]
due to Lemma \ref{compara euclideano com F}.

\

Case 3: $d$ is above $l_{-\hat \theta}(b)$. (that is, compared to Case 1, the length of the parallel sides are relatively $E$-longer than the nonparallel sides).

In this case, $l_{-\hat \theta}(b)$ intercepts $[c,d]$ at an interior point $e$ and $l_{-\hat \theta}(d)$ intercepts $[a,b]$ in an interior point $g$.
Then the trace of $T$ is given by $[a, g, b, c, e, d, a]$ with directions $\left< 8, 8, 6, 2, 2, 10 \right>$ and we have to prove that
$\ell_{\hat F} ([c, e, d]) < \ell_{\hat F} ([d,$ $ a, g, b, c])$.

First of all $\ell_{\hat F} ([e, d]) < \ell_{\hat F} ([g, b])$ holds due to Lemma \ref{lados paralelos}. 

It's worth to notice that if we drop the segments $[e,d]$ and $[g,b]$ from the trapezoid $[a, g, b, c, e, d, a]$ and identify $g$ with $b$ and $d$ with $e$, we get a trapezoid as in Case 1. 
Consequently we have that 
\[
\frac{\ell_E ([d, a, g]) + \ell_E ([b, c])}{\ell_E ([c, e])} = L_{\hat \theta},
\]
what implies that
\[
\ell_{\hat F} ([c, e]) < \ell_{\hat F} ([d, a, g]) + \ell_{\hat F} ([b,c]) 
\]
due to Lemma \ref{compara euclideano com F}.
Thus
\[
\ell_{\hat F} ([c, e, d]) < \ell_{\hat F} ([d, a, g, b, c]).
\]
\end{proof}

\begin{lemma}[Trapezoid sides comparison 3]
\label{trapezio3}
Let $T = [a,b,c,d,a]$ be a trapezoid in $\mathbf S$ with directions $\left< 8, 6, 4, 0 \right>$ or $\left< 4, 6, 8, 0 \right>$.
Then the length of one side is less than the sum of the lengths of the other sides.
\end{lemma}

\begin{proof}

The length of one non-parallel side is less than the sum of the length of the other sides due to Lemma \ref{trapezio1}.
The length of a parallel side is less than the sum of the length of the other sides due to Proposition \ref{dos lados eh maior}.
\end{proof}

\begin{lemma}[Parallelogram comparison 3]
\label{paralelogramo3}
Let $P=[a, b, c, d, a]$ be a parallelogram with directions $\left< 8, 4, 2, 10 \right>$.
Then the length of one side is less than the sum of the lengths of the other sides.
\end{lemma}

\begin{proof}

An upper side is smaller than its parallel lower side due to Lemma \ref{lados paralelos}.
The length of the $E$-smaller lower side is less than than the sum of the lengths of the other sides due to Lemma \ref{paralelogramo2}.
Thus we only have to prove that the length of the lower $E$-larger side is less than the sum of the lengths of the other sides.  

In order to fix ideas, suppose that the $E$-larger lower side of $P$ is $[c,d]$.
The other case is analogous.
Set $e = h_0(c) \cap [a,b]$.
Then $[a, e, c, d, a]$ is a trapezoid with directions $\left< 8, 6, 2, 10\right>$ and Lemma \ref{trapezio2} states that
\[
\ell_{\hat F}([c,d]) < \ell_{\hat F} ([d, a, e, c]).
\]  
But
\[
\ell_{\hat F} ([e,c])< \ell_{\hat F} ([b,c])
\]
due to Proposition \ref{triangulo equilatero}.
Therefore
\[
\ell_{\hat F} ([c,d]) < \ell_{\hat F} ([d, a, b, c]).
\]
\end{proof}

\begin{lemma}[Parallelogram comparison 4]
\label{paralelogramo4}
Let $P=[a, b, c, d, a]$ be a parallelogram with directions $\left< 6, 4, 0, 10 \right>$ or $\left< 8, 6, 2, 0 \right>$.
Then the length of one side is less than the sum of the length of the other sides.
\end{lemma}

\begin{proof}

An upper side is smaller than its parallel lower side due to Lemma \ref{lados paralelos}.
The length of a vertical side is less than the sum of the lengths of the other sides due to Proposition \ref{dos lados eh maior}.
The length of the $E$-smaller lower side is less than than the sum of the lengths of the other sides due to Lemma \ref{paralelogramo2}.
Thus we have only to consider the case where the $E$-larger lower side $\sigma$ isn't parallel to $l_0$ and prove that the length of $\sigma$ is less than the sum of the length of the other sides

Consider the case $\left< 6, 4, 0, 10 \right>$ with $\ell_E([b,c]) > \ell_E ([c,d])$.
The other case is analogous.
We have to prove that
\[
\ell_{\hat F}([b,c]) < \ell_{\hat F}([c,d,a,b]).
\]
Let $e=h_2(b) \cap [a,d]$.
Then $[e, d, c, b, e]$ is a trapezoid with directions $\left< 4, 6, 10, 2\right>$ and 
\[
\ell_{\hat F} ([b,c]) < \ell_{\hat F} ([b, e, d, c])
\]
due to Lemma \ref{trapezio2}.
But
\[
\ell_{\hat F} ([b, e]) < \ell_{\hat F} ([b, a]) + \ell_{\hat F} ([a, e])
\]
due to Proposition \ref{triangulo equilatero}, what implies that
\[
\ell_{\hat F} ([b, c]) < \ell_{\hat F} ([b, a, d, c]).
\]
\end{proof}

\begin{lemma}[Trapezoid sides comparison 4]
\label{trapezio4}
Let $[a,b,c,d,a]$ be a trapezoid in $\mathbf S$ with directions $\left< 4, 8, 10, 0 \right>$ or $\left< 8, 4, 2, 0 \right>$.
Then the length of one side is less than the sum of the length of the other sides.
\end{lemma}

\begin{proof}

We analyse the case $\left< 4, 8, 10, 0\right>$.
The other case is analogous.

The length of a non-parallel side is smaller than the sum of the length of the other sides due to Lemma \ref{trapezio1}.
In addition we have that $\ell_{\hat F} ([a,b]) < \ell_{\hat F} ([b, c, d, a])$ due to Proposition \ref{abaixo eh maior 1}.

Let us prove that $\ell_{\hat F} ([c, d]) < \ell_{\hat F} ([d, a, b, c])$.
Denote $e = h_0(c) \cap [a,b]$. 
The parallelogram $[a, e, c, d, a]$ has directions $\left< 4, 6, 10, 0\right>$ and
\[
\ell_{\hat F} ([c, d]) < \ell_{\hat F} ([d, a, e, c])
\]
due to Lemma \ref{paralelogramo4}.

For the triangle $[e, b, c, e]$,
\[
\ell_{\hat F} ([e, c]) < \ell_{\hat F} ([b, c])
\]
holds due to Proposition \ref{triangulo equilatero}.
Therefore
\[
\ell_{\hat F} ([c, d]) < \ell_{\hat F} ([d, a, e]) + \ell_{\hat F}([b, c])< \ell_{\hat F} ([d, a, b, c]).
\]
\end{proof}

Now we can join the results from Lemmas \ref{trapezio1} to Lemma \ref{trapezio4} in order to present a unique result:

\begin{theorem}
\label{um menor do que tres}
Let $Q = [a, b, c, d, a]$ be a trapezoid or a parallelogram in $\mathbf S$.
Then the length of one of its sides is less than the sum of the length of the other three sides.
\end{theorem} 

The next result helps to replace some preferred paths above $[p,q]_{\text{min}}$ by shorter preferred paths closer to $[p,q]_{\text{min}}$ (see Lemma \ref{segmentos minimizantes}). 

\begin{proposition}
\label{um menor do que tres 2} 
Let $Q = [a_1 , a_2, a_3, \ldots, a_n] \in \mathbf S$ and $k \in \{0, 2, 4\}$ such that 
\begin{enumerate}
\item $a_n \in l_k(a_1)$;
\item $a_3, a_4, \ldots, a_{n-1} \in l_k(a_2)$;
\item $l_k (a_1) \neq l_k(a_2)$.
\end{enumerate}
Then 
\begin{equation}
\label{desigualdade cruzando}
\ell_{\hat F}([a_1,a_n]) < \ell_{\hat F}([a_1, a_2, \ldots, a_n]).
\end{equation}
\end{proposition}

\begin{proof}

If $a_1 = a_n$, then (\ref{desigualdade cruzando}) is trivial. 
If $a_2 = a_{n-1}$, then Proposition \ref{triangulo equilatero} proves (\ref{desigualdade cruzando}). 
Then we can suppose that $a_1 \neq a_n$ and $a_2 \neq a_{n-1}$ and it is enough to prove that
\begin{equation}
\label{desigualdade cruzando simplificado}
\ell_{\hat F}([a_1, a_n])< \ell_{\hat F} ([a_1, a_2, a_{n-1}, a_n]).
\end{equation}
If $[a_1, a_2] \cap [a_{n-1},a_n] = \emptyset$, then $[a_1, a_2, a_{n-1}, a_n, a_1]$ is a parallelogram or a trapezoid and Theorem \ref{um menor do que tres} settles (\ref{desigualdade cruzando simplificado}). 
Otherwise we have that $[a_1,a_2] \cap [a_{n-1},a_n] = \{e\}$, where $e$ is an interior point of $[a_1,a_2]$ and $[a_{n-1},a_n]$. 
The path $[a_1, e, a_n, a_1] \in \mathbf S$ is a triangle and
\[
\ell_{\hat F}([a_1,a_n]) < \ell_{\hat F} ([a_1, e]) + \ell_{\hat F}([e,a_n]) < \ell_{\hat F}([a_1, e, a_2, a_{n-1} , e , a_n]) 
\]
\[
= \ell_{\hat F}([a_1, a_2, a_{n-1}, a_m]),
\]
where the first inequality is due to Proposition \ref{triangulo equilatero}.
\end{proof}

Although we don't use the next theorem in this work, it is worth to present it here.
It states that $Q$ doesn't need to be a parallelogram or a trapezoid in Theorem \ref{um menor do que tres}.
Observe that it also generalizes Proposition \ref{um menor do que tres 2}.

\begin{theorem}
\label{um menor do que tres 3}
Let $[a, b, c, d, a] \in \mathbf S$. 
Then the length of one of its line segments is less than the sum of the length of the other three line segments.
\end{theorem}

\textsc{Idea of the proof.} 
There are several cases to consider, but all of them can be solved easily using Theorem \ref{triangulo equilatero} and Theorem \ref{um menor do que tres}.
The analysis can be made in the following way:
\begin{itemize}
\item Begin with an arbitrary path $[a,b] \in \mathbf S$;
\item Consider all possible directions for $[b,c]\in \mathbf S$. 
For each direction, analyse the placement of $c$ in each of the subsets $S_k(a)$ or $h_k(a)$;
\item Observe that $d \in l_{k_1}(c) \cap l_{k_2}(a)$ for some $k_1, k_2 \in \{0,2,4\}$. 
\end{itemize}
The details of the proof are left to the reader.

\section{$\mathbf S$-minimizing paths}
\label{s-minimizing}

\begin{definition}[$\mathbf S$-minimizing paths]
\label{P-minimizing path}
Let $p,q \in (\hat M, \hat F)$. 
A path $\gamma$ in $\mathbf S$ connecting $p$ and $q$ is a $\mathbf S$-minimizing path if $\ell_{\hat F} (\gamma)\leq \ell_{\hat F} (\tilde \gamma)$ for every $\tilde \gamma\in \mathbf S$ that connects $p$ and $q$.
\end{definition}

In this section we prove that if $p,q \in \hat M$, then $[p,q]_{\text{min}}$ is the unique $\mathbf S$-minimizing path which connects $p$ and $q$.

\begin{lemma}
\label{segmentos minimizantes}
Suppose that $[a,b] \in \mathbf S$.
Then $[a,b]$ is the unique $\mathbf S$-minimizing path that connects $a$ and $b$.
\end{lemma}

\begin{proof}

If $[a,b]$ has direction $\left< 0\right>$ (or $\left< 6\right>$), then Proposition \ref{dos lados eh maior} settles this proposition.

In order to fix ideas, suppose that $[a,b]$ has direction $\left< 4\right>$.
The other cases follow as an direct adaptation of this case.
 
Consider $\gamma_1=[a,c_1,\ldots,c_j,b] \in \mathbf S$.
$\{l_4(c_i),i=1,\ldots,j\}$ is a finite family of lines which are parallel to $v_4$.
Enumerate the lines strictly above $l_4(a)$ by $\zeta_1, \zeta_2, \ldots, \zeta_n$, with $\zeta_{i+1}$ placed strictly above $\zeta_i$ for $i=1,\ldots, n-1$.
If there aren't any lines $l_4(c_i)$ strictly above $l_4(a)$, then Proposition \ref{abaixo eh maior 1} settles this proposition.

The basic idea is replacing $\gamma_1$ by $\gamma_2 \in \mathbf S$ such that: 
\begin{itemize}
\item $\ell_{\hat F} (\gamma_2)<\ell_{\hat F} (\gamma_1)$; 
\item The vertices of $\gamma_2$, strictly above $l_4(a)$, are contained in $\zeta_1 \cup\ldots \cup \zeta_{n}$; 
\item $\gamma_2$ has strictly less vertices in $\zeta_n$ than $\gamma_1$.
\end{itemize}
We iterate this process until we end up with a path without points in $\zeta_n$.
After that, we continue iterating the same process on $\zeta_{n-1}, \ldots, \zeta_1$ until we get a path in $\mathbf S$ with no points in $\zeta_1\cup \ldots \cup \zeta_n$.
Finally this proposition is settled using Proposition \ref{abaixo eh maior 1}.
Let us go to the details.

We begin with $c_k \in \zeta_n$. 
Then there exist a sequence of points $c_{k_1},\ldots , c_k ,$ $ \ldots, c_{k_2}$ such that 
\begin{enumerate}
\item $k_1 \leq k \leq k_2$;
\item $\{ c_{k_1},\ldots , c_k , \ldots, c_{k_2} \} \subset \zeta_n$;
\item $c_{k_1-1},c_{k_2+1} \not\in \zeta_n$.
\end{enumerate}
Set $d_{k_1} = [c_{k_1 - 1},c_{k_1}] \cap \zeta_{n-1}$ and $d_{k_2} = [c_{k_2 + 1},c_{k_2}] \cap \zeta_{n-1}$.
These intersections are non-empty because there isn't any $c_i$ strictly between $\zeta_{n-1}$ and $\zeta_n$.
Then $d_{k_1}, d_{k_2} \in \zeta_{n-1}$, $c_{k_1}, c_{k_2} \in \zeta_k$ and we have that
\begin{equation}
\label{inequacao quadrilatero prova}
\ell_{\hat F} ([d_{k_1},d_{k_2}]) < \ell_{\hat F} ([d_{k_1}, c_{k_1}, c_{k_2}, d_{k_2}]) \leq \ell_{\hat F} ([d_{k_1}, c_{k_1}, \ldots, c_{k_2}, d_{k_2}])
\end{equation}
due to Proposition \ref{um menor do que tres 2}. If we define
\[
\gamma_2 = [a, c_1, \ldots, c_{k_1 - 1}, d_{k_1}, d_{k_2}, c_{k_2 + 1}, \ldots, c_j, b],
\]
we have that $\gamma_2$ satisfies the three conditions stated above because $\gamma_1$ is equivalent to
\[
[a, c_1, \ldots, c_{k_1 - 1}, d_{k_1}, c_{k_1}, \ldots, c_{k_2}, d_{k_2}, c_{k_2 + 1}, \ldots, c_j, b].
\]

As explained before, we can iterate this process until there is no more points strictly above $l_4(a)$ and Proposition \ref{abaixo eh maior 1} settles this proposition.
\end{proof}

\begin{remark}
\label{remain in the convex hull}
In Lemma \ref{segmentos minimizantes}, when we replace $\gamma_1$ by $\gamma_2$, observe that $\gamma_2$ is contained in the $E$-convex hull of $\gamma_1$.
In fact, the points $d_{k_1}$ and $d_{k_2}$, as well as $[d_{k_1}, d_{k_2}]$, are in the $E$-convex hull of $\gamma_1$.
\end{remark}

\begin{proposition}
\label{mpqsminimiza}
Let $p,q \in \hat M$. Then $[p,q]_{\text{min}}$ is the unique $\mathbf S$-minimizing path.
\end{proposition}

\begin{proof}

We can suppose that $[p,q]$ isn't a preferred path due to Lemma \ref{segmentos minimizantes}. 
Let $\gamma_1 = [p=c_1, c_2, \ldots, c_{n-1}, q=c_n] \in \mathbf S$.
In order to fix ideas, we suppose that $q \in S_3(p)$.
The other cases are similar: The general idea is to replace $\gamma_1$ by a shorter curve inside the concave side of $[p,q]_{\text{min}}$ and Proposition \ref{abaixo eh maior} settles this proposition.

First of all we proceed as in Lemma \ref{segmentos minimizantes}:
If $\gamma_1$ contain points strictly above $l_2(p)$, after some iterations, we can replace $\gamma_1$ by $\gamma_{n_1} \in \mathbf S$ such that 
\begin{itemize}
\item $\ell_{\hat F}(\gamma_{n_1}) < \ell_{\hat F} (\gamma_1)$;
\item All points of $\gamma_{n_1}$ are contained in the lower closed half-plane bounded by $l_2(p)$. 
\end{itemize}
Now we consider $\gamma_{n_1}$. If it has points strictly above $l_4(q)$, we proceed as in Lemma \ref{segmentos minimizantes}, and after some iterations we get a path $\gamma_{n_2} \in \mathbf S$ such that 
\begin{itemize}
\item $\ell_{\hat F}(\gamma_{n_2}) < \ell_{\hat F} (\gamma_{n_1})$;
\item All points of $\gamma_{n_2}$ are contained in the lower closed half-plane bounded by $l_4(q)$. 
\end{itemize}
Observe that $\gamma_{n_2}$ is also contained in the lower closed half-plane bounded by $l_2(p)$ due to Remark \ref{remain in the convex hull}.
Therefore $\gamma_{n_2}$ lies in the concave side of $[p,q]_{\text{min}}$ and Proposition \ref{abaixo eh maior} states that \[
\ell_{\hat F} ([p,q]_{\text{min}}) \leq \ell_{\hat F} (\gamma_{n_2}),
\]
and the equality holds iff $\gamma_{n_2} = [p,q]_{\text{min}}$.
Thus
\[
\ell_{\hat F} ([p,q]_{\text{min}}) \leq \ell_{\hat F} (\gamma_1),
\] 
and the equality holds iff $\gamma_1 = [p,q]_{\text{min}}$.
\end{proof}

\section{Length comparison between preferred paths and piecewise smooth paths}
\label{smooth and preferred}

In this section we prove that if $\gamma$ is a piecewise smooth path connecting $p$ and $q$, then 
\[
\ell_{\hat F}([p,q]_{\text{min}}) \leq \ell_{\hat F}(\gamma),
\]
and the equality holds iff $\gamma = [p,q]_{\text{min}}$ (See Theorem \ref{principal}).
We prove some preliminary results before.

\begin{lemma}
\label{finsler flat}
\[
\ell_{F_0} ([p,q]_{\text{\em min}})=\ell_{F_0}([p,q])=F_0(q-p).
\]
\end{lemma}

\begin{proof}

The first equality is straightforward from the definition of $F_0$ and $[p,q]_{\text{min}}$. 
The second equality is due to Proposition \ref{mfo}.
\end{proof}

\begin{lemma}
\label{limite fundamental}
Let $\gamma:[0,1] \rightarrow \hat M$ be a smooth curve. 
Then 
\begin{equation}
\label{eq limite fundamental}
\lim_{t \rightarrow t_0} \frac{\ell_{\hat F}( [\gamma(t),\gamma(t_0)]_{\text{\em min}})}{\vert t - t_0 \vert}= \hat F(\gamma^\prime(t_0)).
\end{equation}
\end{lemma}

\begin{proof}
Notice that

\[
\lim_{t \rightarrow t_0}\frac{\ell_{F_0}([\gamma(t),\gamma(t_0)]_{\text{min}})}{\vert t-t_0\vert}
= \lim_{t \rightarrow t_0}\frac{\ell_{F_0}([\gamma(t),\gamma(t_0)])}{\vert t-t_0\vert}
=  \lim_{t \rightarrow t_0} F_0 \left(\frac{\gamma(t) -\gamma(t_0)}{ t-t_0} \right)
\]
\begin{equation}
\label{F_0}
= F_0(\gamma^\prime (t_0))
\end{equation}
due to Lemma \ref{finsler flat}.

Given an arbitrarily small $\delta >0$, there exist a small neighborhood $I$ of $t_0$ such that $(1 - \delta)f(\gamma(t_0)) \leq f(\gamma(t)) \leq (1 + \delta)f(\gamma(t_0))$ for every $t \in I$. 
Then
\[
(1-\delta)f(\gamma(t_0))\frac{\ell_{F_0} ([\gamma(t),\gamma (t_0)]_{\text{min}})}{\vert t-t_0\vert}
\leq  \frac{\ell_{\hat F}([\gamma(t),\gamma (t_0)]_{\text{min}})}{\vert t- t_0\vert} 
\]
\begin{equation}
\label{sanduiche}
\leq(1+\delta)f(\gamma(t_0))\frac{\ell_{F_0} ([\gamma(t),\gamma (t_0)]_{\text{min}})}{\vert t- t_0\vert}
\end{equation}
for every $t \in I-\{t_0\}$.
Now combining (\ref{F_0}) and (\ref{sanduiche}) we get (\ref{eq limite fundamental}).
\end{proof}

\begin{remark}
After the proof that $[p,q]_{\text{\em min}}$ are minimizing paths (see Theorem \ref{principal}), it will follow that Lemma \ref{limite fundamental} is a particular case of the following theorem:
\end{remark}

\begin{theorem}
\label{velocidade Finsler}
Let $(M,F)$ be a $C^0$-Finsler manifold and $\gamma:(-\varepsilon, \varepsilon )\rightarrow M$ be a smooth path such that $\gamma(0)=p$ and $\gamma^\prime(0)=v$. Then 
\[
\lim_{t\rightarrow 0}\frac{d_F(\gamma(t),p)}{\vert t \vert}=F(v).
\]
\end{theorem}

\begin{proof}

See Theorem 3.7 of \cite{BenettiFukuoka}.

\end{proof}

\begin{lemma}
\label{compara suave local}
Let $\gamma:[0,1] \rightarrow \hat M$ be a smooth path and $t_0 \in [0,1]$ a point such that $\gamma^\prime(t_0)$ isn't a preferred vector. 
Then there exist a neighborhood $I$ of $t_0$ such that 
\[
\ell_{\hat F} (\gamma\vert_{[t_0,t]}) > \ell_{\hat F}([\gamma(t),\gamma(t_0)]_{\text{\em min}})
\] 
for every $t\in I-\{t_0\}$.
\end{lemma} 

\begin{proof}

In order to fix ideas, suppose that $\gamma^\prime(t_0) \in S_3$.
The other cases are analogous. 
Let $I$ a neighborhood of $t_0$ such that $\gamma^\prime(t) \in S_3$ for every $t \in I$.
Consider $s_0 \in I$ and, in order to fix ideas, suppose that $s_0>t_0$ (if $s_0< t_0$, the analysis is analogous).
Observe that 
\[
\beta := [\gamma(t_0),\gamma (s_0)]_{\text{min}} = [\gamma(t_0),a,\gamma(s_0)]
\] 
is strictly above $\gamma\vert_{[t_0,s_0]}$ except at endpoints.
Let $u_0=(t_0 + s_0)/2$ and
consider 
\[
\eta := [\gamma(t_0),\gamma(u_0)]_{\text{min}} \cup [\gamma(u_0),\gamma(s_0)]_{\text{min}} = [\gamma(t_0),b,\gamma(u_0), c, \gamma(s_0)].
\]
Then $\eta$ is also strictly above $\gamma\vert_{[t_0,s_0]}$ except at $\gamma(t_0), \gamma(u_0)$ and $\gamma(s_0)$.
$\beta$ can be written as $[\gamma(t_0), b, a, c, \gamma(s_0)]$ and the difference between $\beta$ and $\eta$ is the parallelogram
$[b, a, c,\gamma(u_0),b]$ with directions $\left< 2, 4, 8, 10 \right>$,
with upper sides
$[c,a,b] \subset \beta$ and lower sides $[b,\gamma(u_0),c]\subset \eta$.
Therefore
\[
\ell(\beta) < \ell(\eta)
\]
due to Corollary \ref{paralelogramo}.
Set 
\[
5\varepsilon = \ell(\eta) - \ell(\beta).
\]

As a consequence of the uniform continuity of $t \mapsto \hat F(\gamma^\prime(t))$, there exist a $\delta >0$ such that
\begin{equation}
\label{aproximaFgammalinha1}
\left\vert \int_{t_0}^{u_0} \hat F(\gamma^\prime(t))dt - \sum_{i=1}^{n_{\tilde{\mathcal P}}} \hat F(\gamma^\prime (\bar \tau_i))(\tilde \tau_i- \tilde \tau_{i-1}) \right\vert < \varepsilon
\end{equation}
and
\begin{equation}
\label{aproximaFgammalinha2}
\left\vert \int_{u_0}^{s_0} \hat F(\gamma^\prime(t))dt - \sum_{j=1}^{n_{\tilde {\mathcal Q}}} \hat F(\gamma^\prime (\bar \nu_j))(\tilde \nu_j- \tilde \nu_{j-1}) \right\vert < \varepsilon
\end{equation}
for every pair of partitions $\tilde{\mathcal P}=\{t_0 = \tilde \tau_0 < \tilde \tau_1 < \ldots <  \tilde \tau_{n_{\tilde{\mathcal P}}} = u_0 \}$ and $\tilde{\mathcal Q} = \{u_0 = \tilde \nu_0 <  \tilde \nu_1 < \ldots <  \tilde \nu_{n_{\tilde{\mathcal Q}}}= s_0 \}$ with norms less than $\delta$ and every choice of $\bar \tau_i\in [\tilde \tau_{i-1},\tilde \tau_i]$ and $\bar \nu_i\in [\tilde \nu_{i-1},\tilde \nu_i]$.

We claim that we can choose two partitions $\mathcal P=  \{t_0 = \tau_0 < \tau_1 < \ldots < \tau_{n_{\mathcal P}} = u_0\}$ and $\mathcal Q = \{u_0 = \nu_0 < \nu_1 < \ldots < \nu_{n_{\mathcal Q}} = s_0 \}$ with norms less than $\delta$ and points $\hat \tau_i \in [\tau_{i-1}, \tau_i]$ and $\hat \nu_i \in [\nu_{i-1}, \nu_i]$ for every $i$ such that
\begin{equation}
\label{aproxima lm1}
\left\vert \sum_{i=1}^{n_{\mathcal P}} \ell_{\hat F}([\gamma(\tau_i),\gamma(\tau_{i-1})]_{\text{min}}) - \sum_{i=1}^{n_{\mathcal P}} \hat F(\gamma^\prime (\hat \tau_i))(\tau_i-\tau_{i-1}) \right\vert < \varepsilon
\end{equation}
and
\begin{equation}
\label{aproxima lm2}
\left\vert \sum_{i=1}^{n_{\mathcal Q}} \ell_{\hat F}([\gamma(\nu_i),\gamma(\nu_{i-1})]_{\text{min}}) - \sum_{i=1}^{n_{\mathcal Q}} \hat F(\gamma^\prime (\hat \nu_i))(\nu_i-\nu_{i-1}) \right\vert < \varepsilon.
\end{equation}
Let us prove (\ref{aproxima lm1}).
For each $\tau \in [t_0, u_0]$, we choose an neighborhood $I_\tau = [t_0, u_0] \cap (\tau -\delta_\tau, \tau + \delta_\tau)$ with $\delta_\tau \leq \delta$ such that
\[
\left\vert \frac{\ell_{\hat F}([\gamma(t),\gamma(\tau)]_{\text{min}})}{\vert t - \tau \vert}- \hat F (\gamma^\prime(\tau)) \right\vert < \frac{\varepsilon}{u_0 - t_0}
\] 
holds for every $t \in I_\tau \backslash \{\tau\}$ (see Lemma \ref{limite fundamental}). 
If $\tau \in (t_0, u_0)$, then we can choose $I_\tau \subset (t_0, u_0)$.
From $\{I_\tau\}_{\tau \in [t_0, u_0]}$, we can choose a finite subcover $\{I_{\tau_k}\}_{k = 0, \ldots, m}$ of $[t_0, u_0]$.
From $\{I_{\tau_k}\}_{k = 0, \ldots, m}$, we can choose a subcover such that $I_{\tau_{k_1}} \not \subset I_{\tau_{k_2}}$ for every $k_1 \neq k_2$.
This latter subcover can be reindexed as \[
\{[\tau_0, \tau_0 + \delta_{\tau_0}),(\tau_2 - \delta_{\tau_2},\tau_2 + \delta_{\tau_2}),(\tau_4 - \delta_{\tau_4}, \tau_4 + \delta_{\tau_4}), \ldots,(u_0-\delta_{u_0},u_0]\},
\] 
with $\tau_{i-2} < \tau_i$ for every $i$.
It is not difficult to see that $\tau_0 < \tau_2-\delta_{\tau_2} < \tau_4 - \delta_{\tau_4} < \ldots < u_0 - \delta_{u_0}$  and $\tau_0 + \delta_{\tau_0} < \tau_2 + \delta_{\tau_2} < \tau_4 + \delta_{\tau_4} < \ldots < u_0$ due to the construction of this latter subcover of $[t_0, u_0]$. 

Now we choose the points $\tau_k$, with ``odd'' $k$. 
$\tau_1 \in (\tau_0,\tau_2)$ is chosen in $[\tau_0,\tau_0 + \delta_{\tau_0}) \cap (\tau_2 - \delta_{\tau_2}, \tau_2 + \delta_{\tau_2})$. 
$\tau_3 \in (\tau_2,\tau_4)$ is chosen in $ (\tau_2 - \delta_{\tau_2},\tau_2 + \delta_{\tau_2}) \cap (\tau_4 - \delta_{\tau_4}, \tau_4 + \delta_{\tau_4}) $ and so on.

Finally we choose $\hat \tau_0= \tau_0, \hat \tau_1=\hat \tau_2 = \tau_2$, $\hat \tau_3 = \hat \tau_4 = \tau_4$, $\ldots$, $\hat \tau_{2i-1} =\hat \tau_{2i} = \tau_{2i}, \ldots$. It is straightforward that (\ref{aproxima lm1}) is satisfied. 

Equation (\ref{aproxima lm2}) is proved analogously.
Therefore
\[
\ell_{\hat F}([\gamma(t_0),\gamma(s_0)]_{\text{min}}) 
= \ell_{\hat F} ([\gamma(t_0), \gamma(u_0)]_{\text{min}}) 
+ \ell_{\hat F} ([\gamma(u_0), \gamma(s_0)]_{\text{min}})-5\varepsilon
\]
\[
\leq \sum_{i=1}^{n_{\mathcal P}} \ell_{\hat F}([\gamma(\tau_i),\gamma(\tau_{i-1})]_{\text{min}}) 
+ \sum_{i=1}^{n_{\mathcal Q}} \ell_{\hat F} ([\gamma(\nu_i),\gamma(\nu_{i-1})]_{\text{min}}) - 5\varepsilon
\]
\[
< \int_{t_0}^{u_0} \hat F(\gamma^\prime(t))dt + \int_{u_0}^{s_0} \hat F(\gamma^\prime(t))dt - \varepsilon < \int_{t_0}^{s_0} \hat F(\gamma^\prime(t))dt,
\]
where the first inequality is due to Proposition \ref{mpqsminimiza} and the second inequality is due to (\ref{aproximaFgammalinha1}), (\ref{aproximaFgammalinha2}), (\ref{aproxima lm1}) and (\ref{aproxima lm2}).
\end{proof}

\begin{proposition}
\label{suave generico nao minimiza}
Let $\gamma:[0,1] \rightarrow \hat M$ be a piecewise smooth path. Suppose that there exist $t_0 \in [0,1]$ such that $\gamma^\prime(t_0) \neq 0$ isn't a preferred vector. 
Then $\ell_{\hat F} (\gamma) > \ell_{\hat F}([\gamma(0),\gamma(1)]_{\text{\em min}})$.
\end{proposition}

\begin{proof}

According to Lemma \ref{compara suave local}, there exist a neighborhood $I$ of $t_0$ such that
\[
\ell_{\hat F}([\gamma(t_0),\gamma(t)]_{\text{min}}) < \ell_{\hat F}(\gamma\vert_{[t_0,t]})
\] 
for every $t \in I-\{t_0\}$.
Fix $s_0 \in I$ (in order to fix ideas, we consider $s_0 > t_0$) and set
\[
5\varepsilon = \ell_{\hat F}(\gamma\vert_{[t_0,s_0]}) - \ell_{\hat F}([\gamma(t_0),\gamma(s_0)]_{\text{min}}).
\]
Now we consider a partition $\mathcal P$ of $[0,t_0]$ and a partition $\mathcal Q$ of $[s_0,1]$ such that (\ref{aproximaFgammalinha1}), (\ref{aproximaFgammalinha2}), 
(\ref{aproxima lm1}) and (\ref{aproxima lm2}) hold.
Then
\[
\ell_{\hat F}([\gamma(0), \gamma(1)]_{\text{min}})
\]
\[
\leq \sum_{i=1}^{n_{\mathcal P}} \ell_{\hat F}([\gamma(t_i),\gamma(t_{i-1})]_{\text{min}}) 
+ \ell_{\hat F}([\gamma(t_0),\gamma(s_0)]_{\text{min}}) 
+ \sum_{i=1}^{n_{\mathcal Q}} \ell_{\hat F}([\gamma(s_i),\gamma(s_{i-1})]_{\text{min}})
\]
\[
< \int_{0}^{t_0} \hat F(\gamma^\prime(t))dt + 2\varepsilon + \int_{t_0}^{s_0} \hat F(\gamma^\prime(t))dt -5\varepsilon + \int_{s_0}^1 \hat F(\gamma^\prime(t)) dt + 2\varepsilon < \ell_{\hat F}(\gamma),
\]
where the first inequality is due to Proposition \ref{mpqsminimiza} and the second inequality is due to Formulas (\ref{aproximaFgammalinha1}), (\ref{aproximaFgammalinha2}), 
(\ref{aproxima lm1}) and (\ref{aproxima lm2}).
\end{proof}

The main theorem of this work follows as a direct consequence of Propositions \ref{mpqsminimiza} and \ref{suave generico nao minimiza}.

\begin{theorem}[Minimizing paths]
\label{principal}
Let $p,q\in (\hat M,\hat F)$. Then $[p,q]_{\text{\em min}}$ is the unique minimizing path connecting $p$ and $q$.
\end{theorem}

In other words, the minimizing paths are $[p,q] \in \mathbf S$ or else $[p,a,q] \in \mathbf S$ with directions $\left< 0,2 \right>$, $\left< 2,4\right>$ or $\left< 4, 6\right>$ (or their reverse paths). 

\begin{theorem}[Geodesics]
\label{principal 2}
The geodesics in $(\hat M, \hat F)$ are subsets of paths $[a, b,$ $c, d, e] \in \mathbf S$ with directions $\left< 0, 2, 4, 6\right>$.
\end{theorem}

\begin{proof}

Let $\gamma$ be a geodesic in $(\hat M, \hat F)$.
Then $\gamma$ can be covered by a finite number of minimizing open subsets, what implies that $\gamma$ can be written as a concatenation of minimizing paths. 
As a consequence, $\gamma$ can be written as a path in $\mathbf S$.
Denote $\gamma = [a_1, a_2, \ldots, a_n] \in \mathbf S$.
If $n\geq 3$, then every $[a_{i-1}, a_i, a_{i+1}] \subset \gamma$ (or its reverse) has directions $\left< 0, 2\right>$, $\left< 2, 4\right>$ or else $\left< 4, 6\right>$.
Therefore $\gamma$ can be written as a subset of $[a,b,c,d,e] \in \mathbf S$ with directions $\left< 0, 2, 4, 6 \right>$.
\end{proof}

\section{Remarks about the geodesic structure of $(\hat M, \hat F)$}
\label{geodesic structure}

In this section we make some remarks about the geodesic structure of $(\hat M, \hat F)$.
We also propose some problems at the end of the section.

Convexity is a complicated issue for general $C^0$-Finsler manifolds because of the lack of the existence and uniqueness of minimizing paths that connect two points, even for small balls.
For instance, we could say that a subset $U$ is convex if for any $p$ and $q \in U$, there exist a minimizing path that connects them and remains in $U$.
Another definition could be: for any pair of points $p,q \in U$, every minimizing path connecting $p$ and $q$ remains in $U$.
Observe that according to the first definition, every Euclidean ball in $(\hat M, F_0)$ is convex because straight lines are minimizing paths.
But they are not convex with respect to the second definition.

The existence and uniqueness of minimizing paths that connects two points of $(\hat M, \hat F)$ allow us to define $\hat F$-convexity and $\hat F$-strong convexity. 
The definition of $\hat F$-strong convexity is borrowed from Riemannian geometry (see \cite{doCarmo2}).

\begin{definition}
\label{define F-convex}
We say that $U \subset \hat M$ is {\em $\hat F$-convex} if for every $p,q \in U$, we have that $[p,q]_{\text{\em min}} \subset U$.
The subset $U$ is {\em $\hat F$-strongly convex} if $[p,q]_{\text{\em min}}\backslash \{p,q\} \subset U$ whenever $p,q \in \bar U$.
\end{definition}

\begin{remark}
\label{F-convexity} 
A lengthy direct verification shows that all the results in this work hold if $\hat M$ is replaced by a $\hat F$-convex open subset of $\hat M$.
\end{remark}

\begin{proposition}
\label{F-convex subsets} Euclidean half-planes $($open or closed$)$ bounded below by Euclidean lines are $\hat F$-convex.
Euclidean half-planes $($open or closed$)$ bounded by preferred lines are $\hat F$-convex.
Intersection of $\hat F$-convex subsets are $\hat F$-convex.
In particular, convex polygons bounded by preferred segments are $\hat F$-convex.
\end{proposition}

\begin{proof}

Suppose that the Euclidean half-plane $H$, bounded by a line $l$, isn't $\hat F$ convex.
Let $p,q \in H$ such that $[p,q]_{\text{min}} \not \subset H$.
Therefore $[p,q]_{\text{min}} = [p,a,q] \in \mathbf S$, $a \not \in H$ and $H$ is placed below $l$.
Notice that there exist an Euclidean line $\tilde l$, parallel to $l$, that intercepts $[p,b,q]$ at two distinct points $c,d \in [p,q]_{\text{min}}\backslash \{p,b,q\}$. 
$[b,c,d,b]$ is a triangle with angle $2\pi/3$ at the vertex $b$.
But preferred triangles are equilateral. 
Therefore $l$ isn't a preferred line, what proves the first two statements.

The proof of the last two statements is trivial.
\end{proof}

\begin{remark}
\label{local} A local version of this theory is straightforward (see Remark \ref{F-convexity}). 
In fact, for every $p \in \hat M$, there exist an arbitrary small neighborhood of $p$ bounded by a preferred parallelogram which is $\hat F$-convex due to Proposition \ref{F-convex subsets}.
\end{remark}

\begin{remark}
\label{nao e Busemann} 
Consider $[a,b,c_1]$ with directions $\left< 0 , 10 \right>$ and $[a,b,c_2]$ with directions $\left< 0, 2 \right>$ such that $\ell_{\hat F}([b,c_1]) = \ell_{\hat F}([b,c_2])$. 
Observe that they are minimizing paths that don't satisfy Condition (4) of Definition \ref{Busemann definition}.
Therefore the manifolds $(\hat M, \hat F)$ aren't a Busemann spaces.
\end{remark}

We end this work presenting another difference between $(\hat M, \hat F)$ and Riemannian manifolds. 

\begin{theorem}
\label{nao e fortemente convexo}
$(\hat M, \hat F)$ doesn't admit a bounded open $\hat F$-strongly subset.
\end{theorem}

\begin{proof}

Suppose that $U$ is a bounded open $\hat F$-strongly subset of $(\hat M, \hat F)$.
We claim for a contradiction. 
First of all, notice that $U$ is a $E$-bounded subset due to Remark \ref{equivalent metrics}.
Let $l_0(a)$ be the first line parallel to $\vec v_0$ that touches the left side of $\bar U$ at $a$ and let $l_2(b)$ be the first line parallel to $\vec v_2$ that touches the upside of $\bar U$ at $b$. 
Then $a=b$, otherwise $[a,b]_{\text{min}} \backslash \{a,b\}$ would be contained in the open subset $U$, what contradicts the fact that there isn't any point of $U$ in the left side of $l_1(a)$ or the fact that there isn't any point of $U$ above $l_2(b)$. 
Now let $l_4(c)$ be the first line parallel to $\vec v_4$ that touches the upside of $\bar U$ at $c$.
For the same reason presented above, we must have $b=c$.
Finally let $l_6(d)$ be the first line parallel to $\vec v_6$ that touches the right side of $\bar U$ at $d$. 
For the same reason presented above, we must have $c=d$. 
Then $U$ is bounded by $l_0(a)$ on its left side and by $l_0(d)=l_0(a)$ on its right side, what gives the contradiction.
\end{proof}

We end this work proposing two problems.

If $p\in \hat M$ and $\vec v \in T_p \hat M$ isn't a preferred vector, then there isn't any geodesic $\gamma^\prime:(-\varepsilon , \varepsilon) \rightarrow (\hat M, \hat F)$ that satisfies $\gamma(0)=p$ and $\gamma^\prime(0)=v$.
If $\vec v$ is a preferred vector, then there are infinitely many minimizing paths that satisfy these conditions because any path $[p,a,q]$ with directions $\left< 0,2 \right>$, $\left<2 , 4\right>$ or $\left< 4, 6\right>$ are minimizing.
Therefore it is meaningless to consider the exponential map.
If $\vec v$ is a positive multiple of $\vec v_6$, then there exist a unique minimizing path $\gamma: [0,\varepsilon ) \rightarrow \hat M$ that satisfy $\gamma (0)=p$ and $\gamma^\prime(0)=v$.
This property doesn't hold for any other preferred direction.
In some sense, $\vec v_6$ is a direction which is more stable than the others.
It will be interesting to study the forward stability of geodesics.

In this work we proved the existence of $C^0$-Finsler structures that admit a large family of metric perturbations that doesn't change its geodesic structure.
However we worked with a ``privileged'' coordinate system, where places with longer paths and shorter paths are clearly identified.
It is worthwhile to study intrinsic geometric invariants that allow us to identify such a kind of geodesic stability on $C^0$-Finsler structures, even locally.

\end{document}